\documentclass[twoside,11pt,reqno]{amsart}

\usepackage{upref,amsxtra,amssymb,amscd}
\usepackage{mathrsfs}  
\usepackage{varioref}
\usepackage{verbatim}
\usepackage{epsfig}
\usepackage{color}
\usepackage{url}

\usepackage{epic,eepic,eucal}
\usepackage{enumerate}

\def\a{         \alpha}

\newcommand{\Top}{\operatorname{top}}

\newcommand{\RR}{{\mathbb R}}

\newcommand{\TT}{{\mathbb T}}

\newtheorem{theo}{\sc Theorem}[section]

\newtheorem{lemm}[theo]{\sc Lemma}
\newtheorem{coro}[theo]{\sc Corollary}

\theoremstyle{definition}

\theoremstyle{remark}

\newtheorem{rema}[theo]{\sc Remark}

\numberwithin{equation}{section}

\usepackage{hyperref}
\usepackage{graphicx} % Required for including images
\usepackage{booktabs} % Top and bottom rules for table
\usepackage{wrapfig} % Allows in-line images
\usepackage[labelfont=bf]{caption} % Make figure numbering in captions bold
\usepackage[top=0.9in,bottom=0.9in,left=1in,right=1in]{geometry} % Reduce the size of the margin

\begin{document}
\title{Nondense orbits for Anosov diffeomorphisms of the $2$-torus}
\author{Jimmy Tseng}
\address{School of Mathematics, University of Bristol, University Walk, Bristol, BS8 1TW UK}
\email{j.tseng@bristol.ac.uk \quad \quad jimmytseng01@gmail.com}

\thanks{The author acknowledges the research leading to these results has received funding from the European Research Council under the European Union's Seventh Framework Programme (FP/2007-2013) / ERC Grant Agreement n. 291147.}

\begin{abstract}   Let $\lambda$ denote the probability Lebesgue measure on $\TT^2$.  For any $C^2$-Anosov diffeomorphism of the $2$-torus preserving $\lambda$ with measure-theoretic entropy equal to topological entropy, we show that the set of points with nondense orbits is hyperplane absolute winning (HAW).  This generalizes the result in~\cite[Theorem~1.4]{T4} for $C^2$-expanding maps of the circle.
\end{abstract}
\maketitle
\section{Introduction}\label{secIntro}  For a dynamical system $f: X \rightarrow X$ on a set $X$ with a topology, we say that a point has \textit{nondense forward orbit} if its forward orbit closure is a proper subset of $X$, and we call the set of these points the {\em nondense set} and denote it by $ND(f)$.  The study of nondense orbits is important for the interactions between ergodic theory and Diophantine approximation and, partly inspired by this, there has been much recent interest in these sets (see \cite{AN, BFK, Da2, Far, MT, T4, Ur, Wu15} for example).  Of particular interest is the size of $ND(f)$ and whether these sets have the winning property, which is a strengthening of having full Hausdorff dimension and is stable under intersections.  (See Section~\ref{secWinning} for the precise definition of the winning property.)  Almost all of the known dynamical systems that have winning nondense sets are (piecewise) linear or come from automorphisms.  A class of exceptions is the result by the author for $C^2$-expanding self-maps on the circle~\cite[Theorem~1.4]{T4}.  There is also a recent result by W.~Wu~\cite[Section~2]{Wu15} for partially hyperbolic diffeomorphisms on compact manifolds in which the diffeomorphisms are conformal  when restricted to the unstable manifolds or the unstable manifolds are one-dimensional.  Wu's result, however, shows only that the nondense set when restricted to an unstable manifold is winning, but does not show that the whole nondense set is winning.\footnote{Wu's result does imply that the whole nondense set has full Hausdorff dimension, which is a significant result.  We are, however, interested in the strengthening to winning because of its stability under intersections.  In particular, our result, Theorem~\ref{thmAnosov2TorusWinningND}, applies to the smaller set of points with nondense {\em complete} orbits because the winning property is stable under intersections and shows that this smaller set also has full Hausdorff dimension.}  General diffeomorphisms present a challenge in dimensions higher than one for determining whether the nondense set is winning or not.

\subsection{Statement of results}  The purpose of this brief note is to show that the members of a large, natural class of {\em nonlinear} self-maps on the $2$-torus have winning nondense sets, thereby, giving a large, natural class of higher dimensional examples of winning nondense sets for more general diffeomorphisms.  Since the winning property is stable under intersections, our results, in fact, also apply to nondense {\em complete} orbits.  Let $\lambda$ denote the probability Lebesgue measure on $\TT^2$.  Our main result is the following.

\bigskip

%*** the conditions on the map are ok; note that $\a>0$ may degrades under the $C^1$ conjugacy, so use HAW from \cite[page 4]{BFKRW}  ******

\begin{theo}\label{thmAnosov2TorusWinningND}  Let $g: \TT^2 \rightarrow \TT^2$ be a $C^2$-Anosov diffeomorphism preserving $\lambda$ such that $h_{\lambda}(g) = h_{\Top}(g)$.  Let $\boldsymbol{x}_0 \in \TT^2$.  Then the set \[ND(g, \boldsymbol{x}_0):= \left\{\boldsymbol{x} \in \TT^2 : \boldsymbol{x}_0 \notin \overline{\{g^n(x) \}_{n=0}^\infty}\right\}\] is hyperplane absolute winning (HAW) and thus has full Hausdorff dimension.
 
\end{theo}

\noindent Note that the hyperplane absolute winning property is a variant of the winning property and implies it (see Section~\ref{secWinning}).

\begin{coro}\label{coroAnosov2TorusWinningND} 
Let $\{g_m\}_{m=0}^\infty$ be a family of $C^2$-Anosov diffeomorphisms preserving $\lambda$ such that $h_{\lambda}(g_m) = h_{\Top}(g_m)$ for all $m$.  Let $\{\boldsymbol{x}_\ell\}_{\ell =0}^\infty \subset \TT^2$.  Then the set \begin{align}\label{eqnIntersectNDs}
 \bigcap_{m=0}^\infty \bigcap_{\ell=0}^\infty \bigg( ND(g_m, \boldsymbol{x}_\ell) \cap ND(g^{-1}_m, \boldsymbol{x}_\ell)\bigg) \end{align} is  hyperplane absolute winning (HAW) and thus has full Hausdorff dimension.
\end{coro}

\begin{rema}
Since the countable intersection property holds for all HAW subsets on $\TT^2$ (see Section~\ref{secWinning}), we may intersect the set from (\ref{eqnIntersectNDs}) with any countable family of HAW subsets on $\TT^2$ and still retain the HAW property and thus also retain the property of having full Hausdorff dimension.  For example, the intersection of the set from (\ref{eqnIntersectNDs}) with any countable collection of translates of badly approximable vectors on $\TT^2$ (which is HAW~\cite[Theorem~2.5]{BFKRW}) is still HAW and, thus, has full Hausdorff dimension.

\end{rema}

\begin{rema}
We can also obtain results when the dimension $d$ is greater than two.  It follows from the proof of Theorem~\ref{thmAnosov2TorusWinningND} and from (4) of Lemma~\ref{lemmWinningProp} that any Anosov diffeomorphism having a bilipschitz (or quasisymmetric)  conjugacy with a hyperbolic automorphism on $\TT^d$ will have an $\a$-winning nondense set for some $\a>0$ (see also Section~\ref{subsecRemarks}).

\end{rema}

\noindent Theorem~\ref{thmAnosov2TorusWinningND} answers the question, first raised in~\cite[Section~6]{T4}, of whether there are nonlinear dynamical systems with winning nondense sets in dimensions greater than one. 

\section{Winning and hyperplane absolute winning}\label{secWinning}  The winning property was introduced by W.~Schmidt~\cite{Sch2} in 1966 and has many later variants (see~\cite{BFKRW, KW, Mc}).  We define the winning property and one particular strengthening, the hyperplane absolute winning (HAW) property, for $\RR^d$.  HAW was introduced in~\cite{BFKRW}.

Let $0 < \alpha <1$ and $0 < \beta<1$.  Let $S \subset \RR^d$ and $\rho(\cdot)$ denote the radius of a closed ball.  Two players, Alice $A$ and Bob $B$, alternate choosing nested closed balls \[B_1 \supset A_1 \supset B_2 \supset A_2 \cdots\] on $\RR^d$ according to the following rules:  \begin{align} \label{eqnWinningRules} \rho(A_n) =  \alpha \rho(B_n) \quad \textrm{ and } \quad \rho(B_n) =  \beta \rho(A_{n-1}).\end{align}  A set $S$ is called \textit{$(\alpha, \beta)$-winning} if Alice has a strategy guaranteeing $\cap_i B_i$ lies in $S$ for the given $\alpha$ and $\beta$.  A set $S$ is called \textit{$\a$-winning} if it is $(\alpha, \beta)$-winning for the given $\alpha$ and every $\beta$.  A set $S$ is called \textit{winning} if it is $\alpha$-winning for some $\alpha$.  These sets have four important properties for us~\cite{Mc, Sch2}: 

% Alice, \textit{wins} if the intersection of these balls lies in $S$.  

\begin{lemm}\label{lemmWinningProp}  Properties of winning sets.
 
\begin{enumerate}
	\item An $\a$-winning set in $\RR^n$ is dense and of full Hausdorff dimension.

\item A countable intersection of $\a$-winning sets is $\a$-winning.

\item An $\a$-winning set in $\RR^n$ with a countable number of points removed is $\a$-winning.

\item The image of an $\alpha$-winning set under a bilipschitz map is $\alpha'$-winning, where $\alpha'$ depends only on $\alpha$ and the map.

\end{enumerate} 
\end{lemm}

%and define the $k$-dimensional $\beta$-absolute game in the following way. 

Now fix a $k \in \{0,1, \cdots ,d - 1\}$ and restrict $0 < \beta < 1/3$. Bob initially chooses $\boldsymbol{x}_1 \in \RR^d$ and $r_1 > 0$ and forms a closed ball $B_1 = B(\boldsymbol{x}_1,r_1)$.  At each stage of the game, after
Bob chooses $\boldsymbol{x}_i \in  \RR^d$ and $r_i > 0$, Alice chooses an affine subspace $L$ of dimension $k$ and removes an $\varepsilon_i$-neighborhood $A_i = L^{(\varepsilon_i)}$ from $B_i := B(\boldsymbol{x}_i , r_i )$ for some $0 < \varepsilon_i \leq \beta r_i$. Then Bob chooses $\boldsymbol{x}_{i+1}$ and $r_{i+1} \geq \beta r_i$ such that\[
B_{i+1} := B(\boldsymbol{x}_{i+1},r_{i+1}) \subset B_i \backslash A_i.\]
A set $S$ is said to be {\em $k$-dimensionally $\beta$-absolute winning} if Alice has a strategy guaranteeing that $\cap_i B_i$ intersects $S$.  A set $S$ is {\em $k$-dimensionally absolute winning} if it is $k$-dimensionally $\beta$-absolute winning for every $0 < \beta < 1/3$.  We call $(d - 1)$-dimensionally absolute winning sets {\em hyperplane absolute winning (HAW) sets}.  These sets have three important properties for us~\cite[Proposition~2.3]{BFKRW}:

\begin{lemm}\label{lemmHAWProp}  Properties of $k$-dimensionally absolute winning sets.
 \begin{enumerate}
\item HAW (and thus $k$-dimensional absolute winning for all $0\leq k \leq d - 1$) implies $\alpha$-winning for all $0<\alpha< 1/2$. 
\item The countable intersection of $k$-dimensionally absolute winning sets is $k$-dimensionally absolute winning.
\item The image of a $k$-dimensionally absolute winning set under a $C^1$ diffeomorphism of $\RR^d$ is $k$-dimensionally absolute winning.

\end{enumerate}

\end{lemm}
\section{Proofs}\label{secProofs}  In this section, we prove Theorem~\ref{thmAnosov2TorusWinningND} and its corollary.  The two main ingredients are the HAW property, already introduced, and the following smooth classification of Anosov diffeomorphisms of the $2$-torus (see~\cite[Corollary~20.4.5]{HKIntroMD} for example):

\begin{theo}\label{thmSmoothClassAno2Torus}
Suppose $g: \TT^2 \rightarrow \TT^2$ is a $C^2$-Anosov diffeomorphism preserving $\lambda$ such that $h_{\lambda}(g) = h_{\Top}(g)$.  Then $g$ is $C^1$ conjugate to a linear automorphism.
\end{theo}

\begin{proof}[Proof of Theorem~\ref{thmAnosov2TorusWinningND}]  The Anosov diffeomorphism $g$ is $C^1$-conjugate to a linear automorphism $T$ by Theorem~\ref{thmSmoothClassAno2Torus}.  Let $h: \TT^2 \rightarrow \TT^2$ be this conjugacy from the Anosov diffeomorphism to the linear automorphism.  The nondense set $ND(T, h(\boldsymbol{x}_0))$ is HAW by~\cite[Theorem~2.6]{BFKRW}.  Consequently, $ND(g, \boldsymbol{x}_0)$ is HAW by (3) from Lemma~\ref{lemmHAWProp}.  This proves the desired result. \end{proof}

\begin{proof}[Proof of Corollary~\ref{coroAnosov2TorusWinningND}]  The desired result follows from the theorem and (1) and (2) from Lemma~\ref{lemmHAWProp} and (1) from Lemma~\ref{lemmWinningProp}. \end{proof}

\subsection{Remarks}\label{subsecRemarks}  A key ingredient in the proof of the smooth classification is the fact that the stable and unstable foliations for a $C^2$-Anosov diffeomorphism of the $2$-torus are $C^1$.  The proof of the fact relies on a more general result, namely that a codimension-$1$ stable manifold for a $C^2$-Anosov diffeomorphism of a compact manifold forms a $C^1$ foliation~\cite[Corollary~4]{HP68}.  The same applies to the unstable manifold, from which we obtain the desired $C^1$ foliations for the $2$-torus.  This leads to a characterization of smooth conjugacies for the $2$-torus~\cite{La87, MaMo1, MaMo2} (also see~\cite[Chapter 20.4.b]{HKIntroMD}) and to our smooth classification.  (Note that, for a torus of dimension three or greater, either the unstable manifold or the stable manifold will not be codimension-$1$.)  

For dimensions three and greater, we lack a smooth classification theorem.  In fact, the characterization of smooth conjugacies is more complicated, and R.~de la Llave showed that two Anosov diffeomorphisms of $\TT^d$ for $d \geq 4$ may be topologically conjugate without being Lipschitz conjugate~\cite[Theorem~6.3]{La92}.  (For the $3$-torus, the situation is not settled.)  A.~Gogolev noted that de la Llave's counterexample can be generalized so that (in most cases) one can take one of the diffeomorphisms to be a linear automorphism~\cite[Theorem~B]{Go08}.  Should a smooth classification in higher dimensions be found, it would be more complicated than Theorem~\ref{thmSmoothClassAno2Torus}, especially since it must allow for cases in which the conjugacy is only H\"older.  On the other hand, we always have a H\"older classification theorem for this context:  every Anosov diffeomorphism of $\TT^d$ is H\"older conjugate to a linear hyperbolic automorphism of $\TT^d$~\cite[Theorems~18.6.1 and~19.1.2]{HKIntroMD}.  This H\"older classification is, however, not strong enough for us because we need the conjugacy to be quasisymmetric or bilipschitz (see~\cite{Mc} and (4) from Lemma~\ref{lemmWinningProp}) to carry out the proof in this paper.  

Another way of approaching this problem for dimensions three and greater that should avoid using a smooth classification is to employ a technique used earlier~\cite{T4, Wu15}, namely using a bounded distortion property (i.e. universal bounds on the ratio of the absolute value of the Jacobian determinant for nearby points---see~\cite[Section~2.1.1]{T4} for example).  If we assume conformality, bounded distortion follows.  However, without this assumption (in dimensions two or greater), bounded distortion may not follow, and without it we may not be able to control the local ``shape'' of successive iterates of small balls.  Since the variants of the game involve intersections with balls, this lack of control is an obstacle that needs to be overcome if this approach were to succeed for more general nonlinear maps.  

Also, we note that for linear maps, one can replace bounded distortion with linear properties of the map.  In particular, one can decompose $\RR^d$ into a direct sum of subspaces appropriate for the linear action and, using the HAW property, one is allowed to remove ``bad'' affine hyperplanes~\cite{BFKRW}.  Linearity seems important to this step and may not possible with a nonlinear system.

Finally, we mention a theorem closely related to our results.  For pairs of toral mappings, one can consider the set of points with orbits that are nondense under one of the pair and dense under the other, a problem first considered in~\cite{BET}.  The nonlinear case for this setup of simultaneous dense and nondense orbits is also interesting, and, for generic pairs of $C^2$-Anosov diffeomorphisms of $\TT^2$, it is known that the desired set is uncountable and dense~\cite[Theorem~1.3]{Tse14}.


\begin{thebibliography}{99}

\bibitem{AN} A. G. Abercrombie and R. Nair, {\em An exceptional set in the ergodic theory of expanding maps on manifolds,} Monatsh. Math. {\bf 148} (2006), 1--17.

%\bibitem{At} J. Athreya, {\em Logarithm laws and shrinking target properties,} preprint, arXiv:0808.0741v1 (2008).

%\bibitem{BBDV} V. Beresnevich, V. Bernik, M. Dodson, and S. Velani, {\em Classical metric Diophantine approximation revisited,} preprint, arXiv:0803.2351v1 (2008).

%\bibitem{BBKM} V. Beresnevich, V. Bernik, D. Kleinbock, and G. Margulis, {\em Metric Diophantine approximation: the Khintchine-Groshev theorem for nondegenerate manifolds,} Mosc. Math. J. {\bf 2} (2002), 203--225. 

\bibitem{BET} V. Bergelson, M. Einsiedler, and J. Tseng, {\em Simultaneous dense and nondense orbits for commuting maps,} Israel J. Math. {\bf 210} (2015), 23--45.

%\bibitem{BBFKW} R. Broderick, Y. Bugeaud, L. Fishman, D. Kleinbock, and B. Weiss, {\em Schmidt's game, fractals, and numbers normal to no base,} preprint, arXiv:0909.4251v2 (2009).

\bibitem{BFK} R. Broderick, L. Fishman, and D. Kleinbock, {\em Schmidt's game, fractals, and orbits of toral endomorphisms,} Ergodic Theory Dynam. Systems {\bf 31} (2011), 1095--1107.

\bibitem{BFKRW} R.~Broderick, L.~Fishman, D.~Kleinbock, A.~Reich, and B.~Weiss, {\em The set of badly approximable vectors is strongly $C^1$ incompressible.}
Math. Proc. Cambridge Philos. Soc. {\bf 153} (2012), no. 2, 319--339. 
%\bibitem{BHKV} Y. Bugeaud, S. Harrap, S. Kristensen, and S. Velani, {\em On shrinking targets for $\ZZ^m$ actions on tori,} preprint, arXiv:0807.3863v1 (2008).

%\bibitem{BS} M. Brin and G Stuck,  ``Introduction to dynamical systems.'' Cambridge University Press, Cambridge, UK, 2002.

%\bibitem{Ca} J. Cassels, ``An Introduction to Diophantine Approximation,''
%Cambridge Tracts in Mathematics and Mathematical Physics {\bf 45}, Cambridge
%University Press, Cambridge, UK, 1957.

%\bibitem{Ch} Y. Cheung, {\em Hausdorff dimension of the set of singular pairs,} Ann. of Math {\bf 173} (2011), 127--167.

%\bibitem{CC} Y. Cheung and N. Chevallier, {\em Hausdorff dimension of the set of singular vectors,} in preparation.
%
%\bibitem{CT} Y. Cheung and J. Tseng, {\em The $s$-power monotone shrinking target property ($s$MSTP) in higher dimensions,} in preparation.

%\bibitem{CK} N. Chernov and D. Kleinbock, {\em Dynamical Borel-Cantelli lemmas for Gibbs measures,} Israel Journal of Mathematics {\bf 122} (2001), 1--27.

%\bibitem{Da} S. G. Dani, {\em On badly approximable numbers, Schmidt games and bounded orbits of flows,} M. M. Dodson and J. A. G. Vickers (eds), {\em Number theory and dynamical systems,} London Mathematical Society Lecture Note Series~{\bf 134}, Cambridge University Press, Cambridge, UK (1989).


\bibitem{Da2} S. G. Dani, {\em On orbits of endomorphisms of tori and the Schmidt game,} Ergodic Theory Dynam. Systems {\bf 8} (1988), 523-529.

%\bibitem{Dod} M. Dodson, {\em Geometric and probabilistic ideas in the metric theory of Diophantine approximations,} Russian Math. Surveys {\bf 48} (1993), 73--102. 


%\bibitem{Do2} D. Dolgopyat, {\em Bounded orbits of Anosov flows}, Duke Math. J. {\bf 87} (1997),  87--114.

%\bibitem{Do} D. Dolgopyat, {\em Limit theorems for partially hyperbolic systems}, Trans. Amer. Math. Soc. {\bf 356} (2004), 1637-1689.

%\bibitem{Dr} C. Dru\c{t}u, {\em Diophantine approximation on rational quadrics}, Math. Ann. {\bf 333} (2005), 405--469. 
%\bibitem{EK} M. Einsiedler and A. Katok, {\em Rigidity of measures--the high entropy case and non-commuting foliations,} Israel J. Math. {\bf 148} (2005), 169--238. 
%
%\bibitem{EL} M. Einsiedler and E. Lindenstrauss, {\em Rigidity properties of $\ZZ^d$-actions on tori and solenoids,}
%Electron. Res. Announc. Amer. Math. Soc. {\bf 9} (2003), 99--110.
%
%\bibitem{EL2} M. Einsiedler and E. Lindenstrauss, {\em On measures invariant under tori on quotients of semi-simple groups,} preprint (2011). Available at \url{http://www.math.ethz.ch/~einsiedl/maxsplit.pdf}

%\bibitem{ET} M. Einsiedler and J. Tseng, {\em Badly approximable systems of affine forms, fractals, and Schmidt games,} to appear in J. Reine Angew. Math.
%
%\url{http://www.reference-global.com/doi/abs/10.1515/CRELLE.2011.078}

%\bibitem{Fal}K. Falconer, ``Fractal geometry:  Mathematical foundations and applications,'' Second ed. John Wiley \& Sons, Inc., Hoboken, NJ, 2003.

\bibitem{Far} D. F\"{a}rm, {\em Simultaneously non-dense orbits under different expanding maps.} Dyn. Syst. {\bf 25} (2010),  531--545. 

%\bibitem{Fa} B. Fayad, {\em Mixing in the absence of the shrinking target property}, Bull. London Math. Soc. {\bf 38} (2006), 829--838.
%
%\bibitem{Fi} L.~Fishman, {\em Schmidt's game, badly approximable linear forms and fractals}, preprint, arXiv:0809.2065v1 (2008).
%
%\bibitem{Fi2} L. Fishman, {\em Schmidt's game on certain fractals}, preprint, arXiv:math/0606298v1 (2006).

%\bibitem{GP} S. Galatolo and P. Peterlongo, {\em Long hitting time, slow decay of correlations and arithmetical properties}, preprint arXiv:0801.3109v2, 2008.

\bibitem{Go08} A.~Gogolev, {\em Smooth conjugacy of Anosov diffeomorphisms on higher-dimensional tori,} J. Mod. Dyn. {\bf 2} (2008), no. 4, 645-700. 


%\bibitem{GS} A. Gorodnik and N. Shah, {\em Khinchin theorem for integral points on quadratic varieties,} preprint, arXiv:0804.3530v1, 2008.



%\bibitem{HV} R. Hill and S. Velani, {\em The ergodic theory of shrinking targets}, Invent. Math., {\bf 119} (1995), 175--198.

\bibitem{HP68} M. W. Hirsch and C. C. Pugh, {\em Stable manifolds and hyperbolic sets}. 1970 Global Analysis (Proc. Sympos. Pure Math., Vol. XIV, Berkeley, Calif., 1968) pp. 133-163 Amer. Math. Soc., Providence, R.I.

%\bibitem{HoJo} R.~Horn and C.~R.~Johnson, ``Matrix analysis.'' 2nd ed. Cambridge University Press, Cambridge, 2013.

%\bibitem{HY14}  H.~Hu and Y.~Yu, {\em On Schmidt's game and the set of points with non-dense orbits under a class of expanding maps.} J. Math. Anal. Appl. {\bf 418} (2014), no. 2, 906--920. 


\bibitem{HKIntroMD} A.~Katok and B.~Hasselblatt, ``Introduction to the modern theory of dynamical systems.  With a supplementary chapter by Katok and Leonardo Mendoza.'' Encyclopedia of Mathematics and its Applications, {\bf 54}. Cambridge University Press, Cambridge, 1995.



%\bibitem{Kh} A. Khinchin, ``Continued Fractions,'' The University of Chicago Press, Chicago, 1964.

%\bibitem{Ki} D. Kim, {\em The shrinking target property of irrational rotations}, Nonlinearity {\bf 20} (2007), 1637--1643.

%\bibitem{Kl} D. Kleinbock, {\em Badly approximable systems of affine forms}, J. Number Theory {\bf 79} (1999), 83 --102. 

%\bibitem{Kl2} D. Kleinbock, {\em Nondense orbits of flows on homogeneous spaces,} Ergodic Theory Dynam. Systems {\bf 18} (1998), 373-396.
%
%\bibitem{KM2} D. Kleinbock and G. Margulis, {\em Bounded orbits of nonquasiunipotent flows on homogeneous spaces}, Amer. Math. Soc. Transl.~{\bf 171} (1996), 141--172.

%\bibitem{KM} D. Kleinbock and G. Margulis, {\em Logarithm laws for flows on homogeneous spaces}, Invent.~Math. {\bf 138} (1999), 451--494.



\bibitem{KW} D. Kleinbock and B. Weiss, {\em Modified Schmidt games and Diophantine approximation with weights}, Adv. Math. {\bf 223} (2010), 1276--1298.
%
%\bibitem{KW2} D. Kleinbock and B. Weiss, {\em Modified Schmidt games and a conjecture of Margulis}, preprint, arXiv:1001.5017v2 (2010).
%\bibitem{KS} K. Krzy\.zewski and W. Szlenk, {\em On invariant measures for expanding differentiable mappings}.  Studia Math.  {\bf 33}  (1969), 83--92.


%\bibitem{Ku} J. Kurzweil, {\em On the metric theory of inhomogeneous Diophantine approximations}, Studia Math. {\bf 15} (1955), 84--112.

%\bibitem{LY1}F. Ledrappier and L.-S. Young, {\em The metric entropy of diffeomorphisms. I. Characterization of measure satisfying Pesin's entropy formula.} Ann. of Math. {\bf 122} (1985), 509-539.
%
%\bibitem{LY2}F. Ledrappier and L.-S. Young, {\em The metric entropy of diffeomorphisms. I.  Relations between entropy, exponents and dimension.} Ann. of Math. {\bf 122} (1985), 540-574.

\bibitem{La87} R.~de la Llave, {\em Invariants for smooth conjugacy of hyperbolic dynamical systems. II.} Comm. Math. Phys. {\bf 109} (1987), no. 3, 369--378. 

\bibitem{La92} R.~de la Llave, {\em Smooth conjugacy and S-R-B measures for uniformly and non-uniformly hyperbolic systems}, Comm. Math. Phys. {\bf 150} (1992), no. 2, 289--320. 

%\bibitem{LLM} R.~de la Llave, J.~M.~Marco, and R.~Moriy—n, {\em Canonical perturbation theory of Anosov systems and regularity results for the Liv\v{s}ic cohomology equation.} Ann. of Math. (2) {\bf 123} (1986), no. 3, 537--611. 



%\bibitem{LM} B. Lytle and A. Maier, {\em Simultaneous dense and nondense points for noncommuting toral endomorphisms}, preprint, arXiv:1407.4014.

\bibitem{MaMo1} J.~M.~Marco and R.~Moriy—n, {\em Invariants for smooth conjugacy of hyperbolic dynamical systems. I.} Comm. Math. Phys. {\bf 109} (1987), no. 4, 681--689. 

\bibitem{MaMo2} J.~M.~Marco and R.~Moriy—n, {\em Invariants for smooth conjugacy of hyperbolic dynamical systems. III.} Comm. Math. Phys. {\bf 112} (1987), no. 2, 317--333. 

\bibitem{MT} B. Mance and J. Tseng, {\em Bounded L\"uroth expansions: applying Schmidt games where infinite distortion exists,} Acta Arith. {\bf 158} (2013), 33--47.

%\bibitem{Ma} G. Margulis. {\em Discrete subgroups and ergodic theory.} In ``Number theory, trace formulas and discrete groups (Oslo, 1987),'' 377--398.  Academic Press, Boston, MA (1989). 

%\bibitem{Mar} J. M. Marstrand.  {\em The dimension of Cartesian products sets,} Proc. Cambridge Philos. Soc. {\bf 50} (1954), 198--202.

\bibitem{Mc} C. McMullen, {\em Winning sets, quasiconformal maps and Diophantine approximation,} Geom. Funct. Anal. {\bf 20} 2010, 726--740.
%{\em http://www.math.harvard.edu/$\sim$ctm/papers/home/text/papers/winning/winning.pdf}

%\bibitem{Ph} W. Philipp, {\em Some metrical theorems in number theory}, Pacific Journal of Mathematics, {\bf 20} (1967), 109--127.

%\bibitem{QX} M. Qian and J.-S. Xie, {\em Entropy formula for endomorphisms:  relations between entropy, exponents and dimension}, Discrete Contin. Dyn. Syst. {\bf 21} (2008), 367--392.

%\bibitem{Sch} W. Schmidt, ``Diophantine approximation,'' Lecture Notes in
%Mathematics~{\bf 785}, Springer-Verlag, Berlin-Heidelberg-New York, 1980.
\bibitem{Sch2} W. Schmidt, {\em Badly approximable numbers and certain games,} Trans. Amer. Math. Soc. {\bf 123} (1966), 178--199.

%\bibitem{Sch3} W. Schmidt, {\em Badly approximable systems of linear forms}, J. Number Theory {\bf 1} (1969), 139--154. 

%\bibitem{Se} D. Serre, ``Matrices. Theory and applications.'' Second edition. Graduate Texts in Mathematics, {\bf 216}. Springer, New York, 2010.

%\bibitem{ST} R.~Shi and J.~Tseng, {\em Simultaneous dense and nondense orbits and the space of lattices,} International Mathematics Research Notices, to appear.


%\bibitem{Su} D. Sullivan, {\em Disjoint spheres, approximation by imaginary quadratic numbers, and the logarithm law for geodesics,} Acta Math. {\bf 149} (1982), 215-237.

%\bibitem{T3} J. Tseng, {\em General shrinking target properties,} unpublished notes (2005).

%\bibitem{T5} J. Tseng, {\em Badly approximable affine forms and Schmidt games,} J. Number Theory {\bf 129} (2009), 3020-3025.
%$\bullet$ \url{http://dx.doi.org/10.1016/j.jnt.2009.05.006}

%\bibitem{T6} J. Tseng, {\em Badly approximable vectors on rational quadratic varieties,} submitted, arXiv:1008.0445v2 (2010). 

%\url{http://arxiv.org/PS_cache/arxiv/pdf/1008/1008.0445v2.pdf}

%\bibitem{T1} J. Tseng, {\em On circle rotations and the shrinking target properties,} Discrete Contin. Dyn. Syst. {\bf 20} (2008), 1111-1122.

%$\bullet$ \url{http://aimsciences.org/journals/displayArticles.jsp?paperID=3148}

%\bibitem{T2} J. Tseng, {\em Remarks on shrinking target properties,} preprint, arXiv:0807.3298v3 (2010).  
%
%\url{http://arxiv.org/PS_cache/arxiv/pdf/0807/0807.3298v3.pdf}


\bibitem{T4} J. Tseng, {\em Schmidt games and Markov partitions,} Nonlinearity {\bf 22} (2009), 525-543. 

\bibitem{Tse14} J.~Tseng, {\em Simultaneous dense and nondense orbits for toral diffeomorphisms,} Ergodic Theory and Dynamical Systems, to appear.  \url{ http://arxiv.org/abs/1406.1970}

%\bibitem{T7} J. Tseng, {\em The structure of nondense orbits}, in preparation.

%$\bullet$ \url{http://iopscience.iop.org/0951-7715/22/3/001}



%\bibitem{Da} S. G. Dani, {\em On badly approximable numbers, Schmidt games and bounded orbits of flows,} M. M. Dodson and J. A. G. Vickers (eds), {\em Number theory and dynamical systems,} London Mathematical Society Lecture Note Series~{\bf 134}, Cambridge University Press, Cambridge, UK (1989).

%\bibitem{Bu} E. Burger, {\em Exploring the number jungle:  a journey into Diophantine analysis,} Student Mathematical Library~{\bf 8}, AMS (2000).
%
%\bibitem{CK} N. Chernov and D. Kleinbock, {\em Dynamical Borel-Cantelli lemmas for Gibbs measures,} Israel Journal of Mathematics {\bf 122} (2001), 1--27.
%
%\bibitem{Do} D. Dolgopyat, {\em Limit theorems for partially hyperbolic systems}, Trans. Amer. Math. Soc.~{\bf 356} (2004), 1637-1689.
%
%\bibitem{Fa} B. Fayad, {\em Mixing in the absence of the shrinking target property}, Bull. London Math. Soc.~{\bf 38} (2006), 829--838.
%
%\bibitem{HV} R. Hill and S. Velani, {\em The ergodic theory of shrinking targets}, Invent. Math.~{\bf 119} (1995), 175--198.
%
%\bibitem{Kh} A. Khinchin, {\em Continued fractions}, The University of Chicago Press, Chicago (1964).
%
%\bibitem{Ki} D. Kim, {\em The shrinking target property of irrational rotations}, preprint.
%
%\bibitem{Kl} D. Kleinbock, {\em Metric Diophantine approximation and dynamical systems}, Course Notes for Math 203b--Spring, 2004, Brandeis University.

%\bibitem{Kl} D. Kleinbock, {\em Nondense orbits of flows on homogeneous spaces}, Ergod. Th. \& Dynam. Sys. {\bf 18} (1998), 373--396.

%
%\bibitem{KM} D. Kleinbock and G. Margulis, {\em Logarithm laws for flows on homogeneous spaces}, Invent.~Math.~{\bf 138} (1999), 451--494.

%
%\bibitem{Ku} J. Kurzweil, {\em On the metric theory of inhomogeneous Diophantine approximations}, Studia Math.~{\bf 15} (1955), 84--112.
%
%\bibitem{Ph} W. Philipp, {\em Some metrical theorems in number
%theory}, Pacific Journal of Mathematics {\bf 20} (1967), 109--127.
%
%%\bibitem{Sch} W. Schmidt, {\em Metrical theorem on the fractional parts of sequences}, Transactions of the
%%American Mathematical Society \textbf{110} (1964), 493--518.
%

%\bibitem{SMV} J. Tseng, {\em A correction for {\em The Hausdorff dimension of the set of points with non-dense orbit under a hyperbolic dynamical system}}, notes (Strong Markov Partition Version).

%\bibitem{Sch} W. Schmidt, {\em Diophantine approximation,} Lecture Notes in
%Mathematics~{\bf 785}, Springer-Verlag, Berlin-Heidelberg-New York (1980).

\bibitem{Ur} M. Urba\'nski, {\em The Hausdorff dimension of the set of points with nondense orbit under a hyperbolic dynamical system,} Nonlinearity {\bf 4} (1991), 385--397.

\bibitem{Wu15}  W.~Wu, {\em Schmidt games and non-dense forward orbits of certain partially hyperbolic systems,} Ergodic Theory and Dynamical Systems, to appear.  \url{http://dx.doi.org/10.1017/etds.2014.136}

%\bibitem{Yo} L.-S. Young, {\em Dimension, entropy, and Lyapunov exponents,} Ergodic Theory Dynam. Systems {\bf 2} (1982), 109-124.

\end{thebibliography}
\end{document}